\documentclass{amsproc}
\usepackage{amssymb}
\usepackage{tikz}
\usepackage{graphicx}
\usepackage{amsmath}
\usepackage{subfigure}
\usepackage{calc}
\usepackage{float}

\newtheorem{theorem}{Theorem}[section]
\newtheorem{lemma}{Lemma}[section]

\newtheorem{remark}{Remark}[section]
\newtheorem{corollary}{Corollary}[section]

\newtheorem{example}{Example}[section]

\newtheorem{conjecture}{Conjecture}[section]
\def\p{\prime}

\def\0{{\bf{0}}}
\def\1{{\mbox{\tiny $ (1) $ }}}
\def\2{{\mbox{\tiny $ (2) $ }}}
\numberwithin{equation}{section}

\begin{document}

\title{Refinable functions with PV dilations.}

\maketitle
\author{Wayne Lawton\footnote{This work was done during my visit in the Department of Mathematics and Statistics at Auburn University in Spring 2016. I thank Professor Richard Zalik for his great hospitality during my stay in the department.}}
\address{Adjunct Professor, School of Mathematics $\&$ Statistics, \\ University of Western Australia, Perth}
\email{wayne.lawton@uwa.edu.au}
\begin{abstract}
A PV number is an algebraic integer $\alpha$ of degree $d \geq 2$ all of whose
Galois conjugates other than itself have modulus less than $1$.
Erd\"{o}s \cite{erdos} proved that the Fourier transform $\widehat \varphi,$ of a
nonzero compactly supported scalar valued function satisfying the refinement equation
$\varphi(x) = \frac{|\alpha|}{2}\varphi(\alpha x) + \frac{|\alpha|}{2}\varphi(\alpha x-1)$
with $PV$ dilation $\alpha,$ does not vanish at infinity so by the Riemann-Lebesgue
lemma $\varphi$ is not integrable.
Dai, Feng and Wang \cite{daifengwang} extended his result to scalar valued solutions of
$\varphi(x) = \sum_k a(k) \varphi(\alpha  x - \tau(k))$ where $\tau(k)$ are integers and
$a$ has finite support and sums to $|\alpha|$.
In (\cite{lawton3}, Conjecture 4.2) we conjectured that their result holds under the weaker assumption that
$\tau$ has values in the ring of polynomials in $\alpha$ with integer coefficients. This
paper formulates a stronger conjecture and provides support for it based on a solenoidal representation of $\widehat \varphi,$ and deep results of Erd\"{o}s and Mahler \cite{erdosmahler};Odoni \cite{odoni} that give lower bounds for the asymptotic density of integers represented by integral binary forms of degree $> 2;$degree $ = 2,$ respectively. We also construct an integrable vector valued refinable function with PV dilation. \\
\end{abstract}
\noindent{\bf 2010 Mathematics Subject Classification : 11D45, 11R06, 42C40, 43A60}
\allowdisplaybreaks
\newcommand{\dotcup}{\ensuremath{\mathaccent\cdot\cup}}
\newcommand{\bigdotcup}{\bigcup \! \! \! \! \!  \cdot \ \ }

\section{Introduction}
\label{intro}
In this paper $\mathbb Z, \mathbb N = \{1,2,3,...\}, \mathbb Q, \mathbb A, \mathcal O, \mathbb R, \mathbb C$
are the integer, natural, rational, algebraic, algebraic integer, real, and complex numbers. For a ring $R,$
$R[X];R[X,X^{-1}]$ is the ring of polynomials; Laurent polynomials with coefficients in $R$ 
in the indeterminant $X.$ If $\alpha \in \mathbb A$
then $\mathbb Q[\alpha]$ equals the algebraic number field generated by $\alpha$ and we define $\mathcal O_{\alpha} =
\mathcal O \bigcap \mathbb Q[\alpha],$ the degree function $d \, : \, \mathbb A \rightarrow \mathbb N,$ and
the trace and norm functions $T; N \, : \mathbb A \rightarrow \mathbb Q.$ Their restrictions to
$\mathcal O$ are integer valued. For $\alpha \in \mathbb A,$ $P_\alpha(X) \in \mathbb Q[X]$ is its minimal degree monic polynomial and $L(\alpha)$ is the least common multiple of the denominators of the coefficients of $P_\alpha(X).$
$\mathcal O \bigcap \mathbb Q = \mathbb Z,$ $\alpha \in \mathbb A \Rightarrow L(\alpha) \alpha \in \mathcal O,$ and
$\alpha \in \mathcal O \Rightarrow P_\alpha(X) \in \mathbb Z[X].$
There exists $B(\alpha) \in \mathbb N$ with
\begin{equation}
\label{Oinclusion}
    \mathbb Z[\alpha] = \mathbb Z + \alpha \mathbb Z \cdots + \alpha^{d(\alpha)-1} \mathbb Z \subseteq \mathcal O_{\alpha}
    \subseteq \frac{1}{B(\alpha)}\mathbb Z[\alpha],
\end{equation}
and hence, since $N(\alpha) \alpha^{-1} \in \mathcal O_{\alpha},$
\begin{equation}\label{alphainclusion}
    N(\alpha)B(\alpha)\alpha^{-1} \in \mathbb Z[\alpha].
\end{equation}
\begin{example}\label{integralbasis}
  If $\alpha \in \mathbb A$ and $P_{\alpha}(X) = X^3 - X^2 - 2X -8 = 0$ Dedekind showed
  $($\cite{dedekind}, pp. 30-32$)$, $($\cite{narkiewicz}, pp. 64$)$ that
  $\{1,\alpha,\alpha(\alpha+1)/2 \}$ is an integral basis for $\mathcal O_{\alpha}.$
  For this $($\ref{Oinclusion}$)$ holds with $B(\alpha) = 2$ and both inclusions are proper.
\end{example}
$\mathbb T = \mathbb R/\mathbb Z; \mathbb T_c = \{w \in \mathbb C:|w| = 1\}$ is the circle group represented additively; multiplicatively. For $x \in \mathbb R$ we define $||x|| = \min_{k \in \mathbb Z} |x-k| \in [0,\frac{1}{2}]$ and observe that $||x + y|| \leq ||x|| + ||y||.$ Since $x + \mathbb Z = y + \mathbb Z \Rightarrow ||x|| = ||y||,$ we can define $|| \ || \, : \, \mathbb T \rightarrow [0,\frac{1}{2}]$ by $||x+\mathbb Z|| = ||x||.$ For $\alpha \in \mathbb R \, \backslash \, [-1,1]$ define its Pisot set
\begin{equation}\label{Omega}
  \Lambda_\alpha = \{\, \lambda \in \mathbb R \, \backslash \, \{0\} \, : \, \lim_{j \rightarrow \infty} ||\lambda \alpha^j|| = 0 \, \}.
\end{equation}
A Pisot–-Vijayaraghavan (PV) number \cite{cassels,pisot} is $\alpha = \alpha_1 \in \mathcal O$ with $d(\alpha) \geq 2$
whose Galois conjugates $\alpha_2,...,\alpha_d$ have moduli $< 1.$
The Golden Mean $\frac{1+\sqrt 5}{2} \approx 1.6180$ has Galois congugate
$\frac{1-\sqrt 5}{2} \approx -0.6180$ so it is a PV number.
\begin{theorem}\label{pisot}$($Pisot, Vijayaraghavan$)$
If $\alpha \in \mathbb A \backslash [-1,1]$ has degree $d \geq 2$ and $\Lambda_\alpha \neq \phi$ then $\alpha$ is a PV number and
\begin{equation}\label{Omega2}
  \Lambda_\alpha = \{ \, \alpha^{m} \mu \, : m \in \mathbb Z, \, \mu \in \mathbb Q[\alpha]\, \backslash \, \{0\}, \, T(\mu \alpha^j) \in \mathbb Z, \, j = 0,...,d-1 \, \}.
\end{equation}
Furthermore, for $\lambda \in \Lambda_\alpha,$ $||\lambda \alpha^j|| \rightarrow 0$ exponentially fast.
\end{theorem}
\begin{proof}
Cassels (\cite{cassels}, Chapter VIII, Theorem 1) gives a simplified version, based on the properties of recursive sequences,
of Pisot's proof in \cite{pisot}. We relaxed the assumption that $\alpha$ is positive since $\alpha$ is a PV number iff
$-\alpha$ is a PV number. The sequence $s(j) = T(\mu \alpha^j), j \geq 0$ satisfies
$s(j) = -c_{d-1} s(j-1) - \cdots - c_0 s(j-d), \ j \geq d,$
where
$P_{\alpha}(X) = X^{d} + c_{d-1} X^{d-1} + \cdots + c_0.$
Then $($\ref{Omega2}$)$ implies that $s$ has values in $\mathbb Z.$
If $\lambda = \alpha^m\mu$ and $\mu = \mu_1, \mu_2,...,\mu_{d}$ are the Galois conjugates of $\mu,$ then
\begin{equation}\label{PVconv}
  ||\lambda \alpha^j|| \leq |\mu \alpha^{j+m} - T(\mu \alpha^{j+m})|
  \leq \sum_{k=2}^{d} |\mu_k| \, |\alpha_k|^{j+m}, \ \ j \geq - m,
\end{equation}
converges to $0$ exponentially fast as $j \rightarrow \infty$ since $|\alpha_{k}| < 1, k = 2,...,d.$
\end{proof}
This paper studies {\it refinable functions}, nonzero
complex scalar or vector valued distributions $\varphi$ satisfying
a refinement equation
\begin{equation}\label{refinement}
  \varphi(x) = \sum_{k = 1}^{\infty} a(k) \varphi(\alpha x - \tau(k)),
\end{equation}
and whose Fourier transform $\widehat \varphi(y) = \int_{-\infty}^{\infty} f(x) e^{-2\p i xy} dx$
is continuous at $y = 0$ and $\widehat \varphi(0) \neq 0.$ Here the {\it dilation} $\alpha \in \mathbb R \, \backslash \, [-1,1],$ the
coefficient sequence $a,$ which is matrix valued for vector valued refinable functions, decays exponentially fast,
and $\tau$ takes values in $\mathbb Z[\alpha,\alpha^{-1}].$ Refinable functions constructed from integer $\alpha \geq 2$
and integer valued $\tau$ include Daubechies' scaling functions used to construct orthonormal wavelet bases \cite{daub}, basis functions
constructed by Cavaretta, Dahmen, and Michelli from stationary subdivision algorithms \cite{CDM}, and multiwavelets constructed from vector valued refinable functions \cite{ghm}.
$($\ref{refinement}$)$ is equivalent to
\begin{equation}\label{Fouriervarphi}
  \widehat \varphi(y) = \widehat a(y\alpha^{-1}) \widehat \varphi(y\alpha^{-1})
\end{equation}
\begin{equation}\label{Fouriera}
  \widehat a(y) = |\alpha|^{-1} \sum_{k = 1}^{\infty} a(k)e^{-2\pi i \tau(k)y}.
\end{equation}
For scalar valued $\varphi,$ $\widehat a(0) = 1,$ for vector valued $\varphi,$
$\widehat a(0) \, \widehat \varphi(0) = \widehat \varphi(0),$ and
\begin{equation}\label{product}
  \widehat \varphi(y\alpha^J) = \left( \prod_{j < J} \widehat a(y\alpha^{j}) \right) \widehat \varphi(0), \ \ J \in \mathbb Z,
\end{equation}
where for matrices the factors move right with decreasing $j.$
\begin{example}\label{boxcar}
We call $\varphi = \chi_{[0,1]}$ the boxcar function. It satisfies $($\ref{refinement}$)$ with $\alpha = 2,$
$a(0) = a(1) = 1,$ $\tau(0) = 0,$ $\tau(1) = 1.$ Then $\widehat a(y) = e^{\pi i y} \cos \pi y$ and
\begin{equation}
    \widehat \varphi(y2^J) =  e^{\pi i y2^J} \frac{\sin \pi y2^J}{\pi y2^J} = e^{\pi y 2^J} \prod_{j < J} \cos (\pi y 2^J).
\end{equation}
$\varphi$ is integrable and $\widehat \varphi$ vanishes at infinity, despite the fact that $\widehat \varphi(\pi \lambda 2^j)$ converges
to $1$ for every $\lambda \in \mathbb Z[2,\frac{1}{2}]$ $($the set of dyadic rational numbers$)$, because for every such $\lambda$ there exists $j \in \mathbb Z$ such that
$\cos (\pi \lambda 2^j) = 0.$
\end{example}
\begin{example}\label{dyadic}
$\alpha = 2,$ $a(k) = \tau(k) = 2^{1-k}, k \in \mathbb N$ gives
the {\it dyadic} function and
$\widehat a(y) = \sum_{k \in \mathbb N} 2^{-k} \exp(-2\pi i y2^{1-k})$
is Bohr almost periodic \cite{bohr} since
\begin{equation}\label{dyadicAP}
  |\widehat a(y + 2^{L-1}) - \widehat a(y)| \leq 2^{-L}, \ \ y \in \mathbb R, \ L \in \mathbb N.
\end{equation}
Furthermore,
$\inf_{y \in \mathbb R} |\widehat a(y)| > 0,$
and
$\lim_{j \rightarrow \infty} \widehat a(\lambda 2^j) = 1$
for every $\lambda \in \mathbb Z[2,\frac{1}{2}].$
Figure \ref{dyadicfig} shows the modulus of $\widehat a$ over $[0,128].$
$\widehat \varphi$ does not vanish at infinity since
$\lim_{j \rightarrow \infty} \widehat \varphi(2^j) \approx 0.2578 + 0.0692i.$
\end{example}
\begin{figure}
    \centering
    \includegraphics[width=5in]{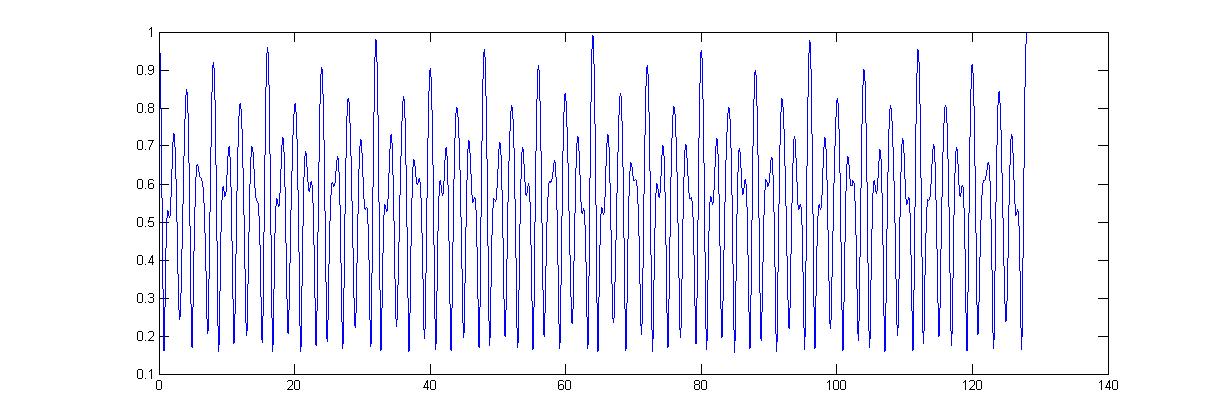}
    \caption{Fourier Transform of the Dyadic Refinement Sequence}
    \label{dyadicfig}
\end{figure}
\begin{example}\label{multi}
If $\alpha = \frac{1+\sqrt 5}{2}$ is the Golden Mean then
$\varphi = \left[ \begin{array}{c}
  \chi_{[0,\alpha^{-1}]} \\
  \chi_{[0,1]}
\end{array} \right]$
satisfies $($\ref{refinement}$)$ with
$a(1) = \left[ \begin{array}{cc}
          0 & 1 \\
          0 & 1
        \end{array} \right],
$
$a(2) = \left[ \begin{array}{cc}
          0 & 0 \\
          1 & 0
        \end{array} \right],
$
$\tau(1) = 0, \tau(2) = 1.$
\begin{equation}\label{multiFa}
  \widehat a(y) = \alpha^{-1} \left[ \begin{array}{cc}
    0 & 1 \\
    e^{-2\pi i y} & 1
        \end{array} \right], \ \ \widehat \varphi(0) = \left[ \begin{array}{c}
    \alpha^{-1} \\
    1
        \end{array} \right].
\end{equation}
$\widehat a(y)$ is never zero and satisfies
$\lim_{j \rightarrow \infty} \widehat a(\lambda \alpha^j) = \widehat a(0)$ for all $\lambda \in \mathbb Z[\alpha,\alpha^{-1}],$
and
\begin{equation}\label{multiprod}
  \widehat \varphi(y\alpha^J) = \left[ \begin{array}{c}
  e^{-\pi i y\alpha^{J-1}} \frac{\sin (\pi y \alpha^{J-1})}{\pi y \alpha^J} \\
  e^{-\pi i y\alpha^{J}} \frac{\sin (\pi y \alpha^{J})}{\pi y \alpha^J}
\end{array} \right] =
\left( \prod_{j < J} \alpha^{-1}
 \left[ \begin{array}{cc}
    0 & 1 \\
    e^{-2\pi i y \alpha^{j}} & 1
        \end{array} \right] \right)
 \widehat \varphi(0).
\end{equation}
\end{example}
This example is related to the wavelets with PV dilations constructed by Gazeau, Patera and
Spiridinov \cite{gazeauspatera}, \cite{gazeauspiridonov} and multiresolution analyses on quasicrystals \cite{lawton3}.
For $\lambda \in \mathbb Z[\alpha,\alpha^{-1}],$ $\widehat \varphi (\lambda \alpha^j)$ is close to the eigenspace of $\widehat a(0)$ with eigenvalue $\alpha; -\frac{1}{\alpha}$ for $j \approx -\infty; \infty.$ This fact makes $\phi$ integrable despite $\widehat a$ never vanishing, in contrast to scalar refinable functions with PV dilations discussed below.
\\ \\
Erd\"{o}s \cite{erdos} proved that if $\alpha$ is a PV number then the refinable function satisfying
\begin{equation}\label{Bernoulli}
  \varphi(x) = \frac{|\alpha|}{2}\varphi(\alpha x) + \frac{|\alpha|}{2}\varphi(\alpha x -1)
\end{equation}
(a Bernoulli convolution) is not integrable by showing that
\begin{equation}\label{BF}
  \widehat \varphi(\alpha^J) = e^{-\pi i \frac{\alpha^J}{\alpha-1}} \prod_{j < J} \cos (\pi \alpha^j)
\end{equation}
fails to converge to $0$ as $J \rightarrow \infty$ We give his proof.
Since $||\alpha^j||$ converges to $0$ exponentially fast as $|j| \rightarrow \infty,$
$\cos \pi \alpha^j$ converges to $\pm 1$ exponentially fast and hence
$\varphi(\alpha^J)$ converges to $0$ iff there exists $j \in \mathbb Z$ such that $\cos \pi \alpha^j = 0,$
or equivalently if $\alpha^j \in \frac{1}{2} + \mathbb Z.$ This is impossible because if $\alpha^j \in \mathbb Q$
then the Galois conjugates of $\alpha$ would be multiples of each other by roots of a cyclotomic polynomial and thus
have identical moduli.
Dai, Feng and Wang \cite{daifengwang} extended Erd\"{o}s' result to scalar valued
refinable functions $\varphi$ that satisfy $($\ref{refinement}$)$ where $a$ has finite support,
$\alpha$ is a PV number, and $\tau$ is integer valued. We give their proof. Since
$\tau$ is integer valued, $\widehat a$ in $($\ref{Fouriera}$)$ has period $1.$ Since $a$ has
finite support, $\widehat a$ is real analytic so the set of its zeros in $[0,1)$ is a finite set $F,$
so the set of zeros of $\widehat a$ equals $F + \mathbb Z.$ If $\widehat \varphi$ vanished at infinity
then for every $m \in \mathbb N,$ $\lim_{J \rightarrow \infty} \widehat \varphi(m \alpha^J) = 0$ and since
$||m\alpha^j||$ converges to $0$ exponentially fast as $|j| \rightarrow \infty,$ the same argument used to obtain
Erd\"{o}s' result implies that there exists $j_m \in \mathbb Z$ such that $m\alpha^{j_m} \in F + \mathbb Z.$
Let $d = d(\alpha).$ $($\ref{Oinclusion}$)$ implies that there exist unique $k_m \in \mathbb Z$ and
$\beta_m \in \alpha \mathbb Z + \cdots + \alpha^{d-1} \mathbb Z$ with
\begin{equation}\label{pos}
  m \alpha^{j_m} = m k_m + m \beta_m, \ \ j_m \geq 0
\end{equation}
and $($\ref{alphainclusion}$)$ implies that
\begin{equation}\label{neg}
  m \alpha^{j_m} = (N(\alpha)B(\alpha)))^{j_m} m k_m + (N(\alpha)B(\alpha)))^{j_m} m \beta_m, \ \ j_m < 0.
\end{equation}
Since $\alpha^{j_m} \notin \mathbb Q,$ $\beta_m \neq 0,$ choosing a prime $p$ that does not divide $N(\alpha)B(\alpha)$
and letting $m = p^n, n \in \mathbb N$ produces a set of values of $(N(\alpha)B(\alpha)))^{j_m} m \beta_m$ that are infinite
modulo $\mathbb Z$ (infinite modulo $\mathbb Q$ since $1,\alpha,...,\alpha^{d-1}$ are rationally
independent), contradicting the fact that $F$ is finite.
\\ \\
This argument holds under the weaker assumption that $a$ decays exponentially fast because then $\widehat a$ is a real analytic function with period $1.$ The remainder of this paper provides support for the following extension of these results:
\begin{conjecture}\label{conj}
  If $\phi$ is a scalar valued refinable function satisfying $($\ref{refinement},$)$
   $\alpha$ is a PV number, $a$ decays exponentially fast, and $\tau$ has values in $\mathbb Z[\alpha,\alpha^{-1}],$
  then $\widehat \varphi$ does not vanish at infinity and hence $\varphi$ is not integrable.
\end{conjecture}
\section{Solenoidal Representation}
\label{sol}
Let $\mathbb T^{\mathbb Z}$ be the group of functions $g : \mathbb Z \rightarrow \mathbb T$
under pointwise addition and the product topology. Tychonoff's theorem \cite{tychonoff} implies
that $\mathbb T^{\mathbb Z}$ is compact.
Define the {\it shift automorphism} $\sigma : \mathbb T^{\mathbb Z} \rightarrow \mathbb T^{\mathbb Z}$ by
%
  $(\sigma g)(j) = g(j+1),\ \ g \in \mathbb T^{\mathbb Z},\ j \in \mathbb Z,$
%
and homomorphisms $\rho_n : \mathbb T^{\mathbb Z} \rightarrow \mathbb T^n$ by
%
  $\rho_n(g) = [g(0),...,g(n-1)]^T, \ \ g \in \mathbb T^{\mathbb Z}, \ n \in \mathbb N.$
%
Let $\alpha$ be a PV number of degree $d = d(\alpha).$ Its Galois conjugates
$\alpha = \alpha_1,...\alpha_d$ are the roots of its
minimal polynomial $P_\alpha(X) = X^d + c_{d-1} X^{d-1} + \cdots + c_0 \in \mathbb Z[X].$
Define the Vandermonde $V$, Frobenius Companion $C$, and Diagonal $D$, matrices
$$
V = \left[\begin{array}{ccc}
     1                & \hdots &  1 \\
     \alpha_1         & \hdots & \alpha_{d} \\
     \vdots           & \hdots & \vdots \\
     \alpha_{1}^{d-1} & \hdots & \alpha_{d}^{d-1} \\
    \end{array}\right]
C = \left[\begin{array}{cccc}
        0 &  1 & \hdots & 0 \\
        \vdots & \ddots & \ddots &  0 \\
        0 &  \hdots & 0 & 1 \\
        -c_0 & \hdots & \hdots & -c_{d-1}
    \end{array}\right]
D = \left[\begin{array}{ccc}
        \alpha_1 & \hdots & 0 \\
        \vdots & \ddots & \vdots \\
        0 &  \hdots & \alpha_d
    \end{array}\right]
$$
$\det V = \prod_{i < j} (\alpha_j - \alpha_i) \neq 0 \Rightarrow V$ is invertible, $CV = VD,$ and $V^{-1}C = DV^{-1}.$
Let
$S_{-} = \{  [0,s_2,...,s_d]^T : s_j \in \mathbb C, \ \alpha_j = \overline \alpha_k \Rightarrow s_j = \overline s_k  \},$
$S_{+} = \{  [s_1,0,...,0]^T  :  s_1 \in \mathbb R  \},$ $S = S_{+} + S_{-}.$ These are real subspaces of $\mathbb C^d$ with
dimensions $d-1,1,d.$
Construct homomorphisms $\vartheta : S \rightarrow \mathbb T^\mathbb Z$ and $\theta : \mathbb R \rightarrow \mathbb T^\mathbb Z,$
\begin{equation}\label{vartheta}
   \vartheta(s)(j) = \sum_{k=1}^{d} s_k \, \alpha_k^j + \mathbb Z, \ \  s \in S, \ j \in \mathbb Z,
\end{equation}
\begin{equation}\label{theta}
   \theta(y)(j)= y \, \alpha^j + \mathbb Z , \ \ y \in \mathbb R, \ j \in \mathbb Z,
\end{equation}
and $\sigma$ invariant subgroups $M_{\pm} = \vartheta(S_{\pm}),$ $M = M_{+} + M_{-},$ and $G = \overline M.$
Then $DS_{\pm} = S_{\pm},$ $S = V^{-1}R^d,$ $M = \vartheta(S),$ $\rho_d(M) = \mathbb T^d,$
$\rho_d(G) = \mathbb T^d,$ and hence $G$ is a compact, connected, abelian group with dimension $\geq d,$
the restriction $\sigma : G \rightarrow G$ gives a dynamical system $(G,\sigma)$ that is an {\it extension} \cite{furstenberg}
of the dynamical system $(\mathbb T^d,C)$ since
$C \circ \rho_d = \rho_d \circ \sigma.$ $0 \in G$ is an equilibrium point of $\sigma$ and
$M_{+};M_{-}$ is the unstable;stable manifold of $0.$
Define $K =$ kernel of $\rho_d : G \rightarrow \mathbb T^d,$
$K_{-} = \bigcup_{j \in \mathbb Z} \sigma^j (K),$ $G_{+} = M_{+},$ $G_{-} = M_{-} + K_{-},$
and $H = G_{+} \cap G_{-}.$ Then $G = M + K = G_{+} + G_{-}$ and the orbits of points in $H$ are {\it homoclinic}
since $h \in H \Rightarrow \lim_{j \rightarrow \pm \infty} \sigma^j h = 0.$
\begin{remark}\label{rem1}
  Following Lind and Ward $($\cite{lindward}, p. 411$)$ we define a {\it solenoid} to be a compact, connected, dimension
$n < \infty$ abelian group. Its dual group $($see Section \ref{app}$)$ is a discrete, torsion free,
rank $n$ abelian group, or equivalently, a subgroup of $\mathbb Q^n.$ The dynamical system $(G,\sigma)$
is {\it expansive}, as defined by Lam \cite{lam}, if there exists a neighborhood $U$ of $0$ such that
$\bigcap_{j \in \mathbb Z} \sigma^j(U) = \{0\}.$
In \cite{lawton1} we used {\it topological entropy} and Pontryagin duality
to prove that every group that admits an expansive automorphism is a
solenoid. Schmidt (\cite{schmidt}, Chap. III) used algebraic methods
to characterize more general expansive transformation groups.
\end{remark}
\begin{theorem}\label{Gl}
The assumptions above imply that:
\begin{enumerate}
    \item $G$ is isomorphic to a group extension of $\mathbb T^d$ by $K.$
    \item If $|c_0| = 1$ then $K = \{0\}$ and $G$ is isomorphic to $\mathbb T^d.$
    \item If $|c_0| > 1$ then $K$ is homeomorphic to Cantor's set.
    \item $G$ is a $d$-dimensional solenoid and $(G,\sigma)$ is expansive.
    \item $\theta(\mathbb R) = \vartheta(S_{+})$ is dense in $G.$
    \item $\theta(\Lambda_\alpha) = H.$
    \end{enumerate}
\end{theorem}
\begin{proof}
(1) holds since $\rho_d : G \rightarrow \mathbb T^d$ is surjective.
(2-3) hold since $($\ref{vartheta}$)$ $\Rightarrow$ every $g \in G_0,$
hence every $g \in G,$ satisfies
$$
  (P_{\alpha}(\sigma)g)(k) = c_0g(k) + \cdots + c_{d-1}g(k+d - 1) + g(k+d) = 0, \ \ k \in \mathbb Z,
$$
so $K = \{ \, g \in \mathbb T^\mathbb Z \, : \, g(k) = 0, k \geq 0 \mbox{ and } c_{0}^{-k}g(k) = 0, k < 0\}.$
(4) holds since $G = G_{+} + G_{-}.$
(5) holds since $(\sigma \theta(x))(j) = \theta(x)(j+1) = \theta(x\alpha) \Rightarrow \theta(\mathbb R)$ is $\sigma$-invariant
and $\rho_d(\theta(\mathbb R)) = R [1,\alpha,....,\alpha^{d-d}]^T + \mathbb Z^d$ is dense in $\mathbb T^d$ by the Kronecker-Weyl theorem \cite{weyl},(Appendix, Theorem \ref{Weyl}) since the entries of $[1,\alpha,...,\alpha^{s-1}]^T$ are rationally independent.
(6) holds since $\theta(s_1) \in H \Leftrightarrow \theta(s_1) \in M_{-} + K_{-}$
iff there exist $-[0,s_2,...,s_d]^T \in S_{-}, m \in \mathbb Z$ with
$\theta(s_1) + \vartheta([0,s_2,...,s_d]^T)\in \sigma^{m}K
\Leftrightarrow \rho_d(\sigma^{-m}\vartheta(s)) = 0$ where $s = [s_1,s_2,...,s_d]^T.$
Theorem \ref{pisot} implies that this condition holds iff
$VD^{-m}s = [T(\alpha^{-m} s_1),...,T(\alpha^{-m+d-1}s_1)]^T \in \mathbb Z^d \Leftrightarrow
s_1 \in \Lambda_\alpha.$
\end{proof}
Assume that $\alpha$ is a PV number of degree $d = d(\alpha),$ $a : \mathbb N \rightarrow \mathbb C$ decays exponentially fast, $\sum_{k \in \mathbb N} a(k) = 1,$ and $\tau \, : \, \mathbb N \rightarrow \mathbb Z[\alpha,\alpha^{-1}].$ For $k \in \mathbb N$ define
$\tau_k : \mathbb Z \rightarrow \mathbb Z$
such that
$
  \tau(k) = \sum_{j \in \mathbb Z} \tau_k(j)\alpha^j.
$
Each $\tau_k$ is finitely supported.
Construct $A : G \rightarrow \mathbb C$
\begin{equation}\label{A}
  A(g) = |\alpha|^{-1} \sum_{k \in \mathbb N} a(k) \exp \left(-2\pi i\sum_{j \in \mathbb Z} \tau_k(j) g(k)\right), \  \ g \in G.
\end{equation}
\begin{theorem}\label{solrep}
  These assumptions give the solenoidal representation $\widehat a = A \circ \theta.$
\end{theorem}
\begin{proof}
Follows from $($\ref{Fouriera}$)$ and $($\ref{theta}$).$
\end{proof}
\section{Zero Sets}
\label{zero}
This section develops relationships between and properties of the zero sets
$\mathcal S(\widehat \phi),$ $\mathcal S(\widehat a),$ and $\mathcal S(A).$
Since $a$ has exponential decay $A$ is a real-analytic function so $\mathcal S(A)$ is a real-analytic set.
Theorem \ref{solrep} implies $\theta(\mathcal S(\widehat a)) \subset \mathcal S(A)$ and
$($\ref{Fouriervarphi}$)$ implies that
\begin{equation}
\label{SFphi}
    \mathcal S(\widehat \varphi) = \alpha \mathcal S(\widehat a) + \alpha \mathcal S(\widehat \varphi),
\end{equation}
($+$ of zero multiplicities)
so the lower $\underline d$ and upper $\overline d$ asymptotic densities satisfy
\begin{equation}\label{SFphi}
    \underline d(\mathcal S(\widehat a)) = (|\alpha|-1)\, \underline d(S(\widehat \varphi)) \leq
    \overline d(\mathcal S(\widehat a)) = (|\alpha|-1)\, \overline d(S(\widehat \varphi)).
\end{equation}
\begin{theorem}
\label{upperbound}
    If $A$ is not zero then $\overline d(\mathcal S(\widehat a)) < \infty.$
\end{theorem}
\begin{proof}
$\mathcal H(G) = \{\mbox{ closed subsets of } G \, \}$ with the Hausdorff topology is a compact
space (\cite{fell}, p. 205, 255) and $\mathcal T(G) = \{\theta([0,1]) + g:g\in G\}$
is a closed subset and hence compact. If $\overline d(\mathcal S(\widehat a)) = \infty$
then there exist $g \in G$ and a sequence $x_n \in \mathbb R, \ n \in \mathbb N$ such that $A$ has
at least $n$ zeros in $\theta([x_n,x_n+1])$ and $\lim_{n \rightarrow \infty} \theta(x_n) = g.$ Define
$f(x) = |A(\theta(x) + g)|^2, \ x \in \mathbb R.$ Then either $f$ has an infinite number of
zeros in the interval $[0,1]$ or Rolle's theorem implies that it has a zero of infinite order in $[0,1].$ Since
$A$ and hence $f$ is real-analytic, $f = 0.$ Since $\theta(\mathbb R),$ and hence
$\theta(\mathbb R) + g,$ is dense in $G,$ $A$ is zero thus giving a contradicting.
\end{proof}
For $m \in \mathbb Z$ and $\epsilon = [\epsilon_2,...,\epsilon_d]$ satisfying $\epsilon_k > 0$ and $\alpha_j = \overline \alpha_k \Rightarrow \epsilon_j = \epsilon_k$ define
\begin{equation}\label{Udef}
  U(m,\epsilon) = \sigma^m \, K + \{ \, \vartheta(s) \, : \, s \in S_{-} \mbox{ and } |s_k| < \epsilon_k, \ k = 2,...,d \, \} \subset G_{-}.
\end{equation}
Since $U(m,\epsilon)$ decreases to $\{0\}$ as $m \rightarrow \infty$ and $\epsilon_k \rightarrow 0$ we can, and will, choose
$m$ and $\epsilon$ so that $A$ never vanishes on $U(m,\epsilon).$ Also, $U(m,\epsilon)$ is $\sigma$ invariant since
$\sigma U(m,\epsilon) = U(m+1,[\epsilon_2|\alpha_2|,...,\epsilon_d|\alpha_d|]) \subset U(m,\epsilon).$
Let $a$ be the number of real $\alpha_k, k = 2,...,d$ and $b$ be the number of complex conjugate pairs of $\alpha_k, k = 2,...,d,$
and let $\gamma = |\det V|\, |c_0|^{-m} \, 2^a \pi^b\, \epsilon_1 \cdots \epsilon_d.$
For $L > 0$ define
$$W(L) = \{ \, VD^{-m}[y,s_2,...,s_d]^T \, : \, y \in (-L,L), |s_k| < \epsilon_k, \, k = 2,...,d \, \},$$
$$Y(L) = \{ \, y \in (-L,L) \, : \, \theta(y) \in U(m,\epsilon) \, \}.$$
The sets $W(L), \, L > 0$ are convex cylinders parallel to the vector $[1,\alpha,...,\alpha^{d-1}]^T$  whose entries are rationally
independent, therefore
\begin{equation}
\label{card}
    \lim_{L \rightarrow \infty} \frac{1}{2L} \mbox{ card } W(L) \cap \mathbb Z^d = \frac{1}{2L} \mbox{vol } W(L) = \gamma.
\end{equation}
\begin{lemma}
\label{bijection}
\label{same}
The $\mathbb R$-linear function $\xi:W(L)\, \cap\, \mathbb Z^d \rightarrow (-L,L)$ defined by
\\ $\xi(w) = [1,0,...,0]\,  D^mV^{-1}w$ is a bijection onto $Y(L).$ Therefore
\begin{equation}\label{Y}
 \lim_{L \rightarrow \infty} \frac{1}{2L} \mbox{ card } Y(L) = \gamma.
\end{equation}
\end{lemma}
\begin{proof}
Assume $w \in W(L)\, \cap \, \mathbb Z^d.$ If $\xi(w) = 0$ then $w \in VD^{-m}S_{-}$ and hence
$\lim_{\, \ell \rightarrow \infty} C^\ell w = 0.$ Since $w \in \mathbb Z^d,$ $w = 0,$ so
$\xi$ is injective. Assume that $y \in (-L,L).$ Then $\theta(y) \in U(m,\epsilon)$ iff
there exist $-s_k$ that satisfy $|s_k| < \epsilon_k, k = 2,...,d$ and
$$\theta(y) \in \sigma^j K + \vartheta([0,-s_2,...,-s_d]^T) \Leftrightarrow
\sigma^{-m}\vartheta([y,s_2,...,s_d]^T) \in K$$
$$\Leftrightarrow \rho_d(\sigma^{-m}\vartheta([y,s_2,...,s_d]^T)) = 0 \Leftrightarrow
VD^{-m}[y,s_2,...,s_d]^T \in \mathbb Z^d.$$
Since $VD^{-m}[y,s_2,...,s_d]^T \in W(L),$ the last inclusion holds iff $y \in \xi(W(L) \cap \mathbb Z^d).$
This shows that $\xi$ maps $W(L) \cap \mathbb Z^d$ onto $Y(L).$ Then $(\ref{card}) \Rightarrow (\ref{Y}).$
\end{proof}
\begin{theorem}
\label{lowerbound}
  If $\widehat \varphi$ vanishes at infinity then $\underline d(\mathcal S(\widehat \varphi)) \geq \gamma.$
\end{theorem}
\begin{proof}
$($\ref{Y}$)$ implies that it suffices to show that for every $L > 0$ it suffices to show that
$\widehat \varphi$ equals $0$ at every point in $Y(L).$ Since $\widehat \varphi$ vanishes at infinity,
$$0 = \lim_{J \rightarrow \infty} \widehat \varphi (\alpha^J y) =
\lim_{J \rightarrow \infty} \widehat \varphi(y) \prod_{j = 1}^J \widehat a(\alpha^j y).$$
Since $y \in Y(L)$ and $\j \geq 1$ implies $\theta(\alpha^j y) = \sigma^j \theta(y) \in U(m,\epsilon),$
and since $A$ never vanishes on $U(m,\epsilon),$ $($\ref{solrep}$)$ implies that
$\widehat a(\alpha^j y) = A(\sigma^j(\theta(y)) \neq 0.$
Since $\sigma^j(\theta(y))$ converges to $0$ exponentially fast, $\prod_{j = 1}^\infty \widehat a(\alpha^j y) \neq 0,$
and hence $\widehat \varphi(y) = 0.$
\end{proof}
\begin{corollary}
\label{SA}
 If $\widehat \varphi$ vanishes at infinity then $\underline d(\mathcal S(\widehat a)) \geq (|\alpha|-1)\, \gamma.$
 $\mathcal S(A)$ is a union of embedded manifolds in $G$ and has dimension  $d-1.$
\end{corollary}
\begin{proof}
The first assertion follows from $($\ref{SFphi}$).$ Since $\mathcal S(A)$ is an real-analytic set it is homeomorphic to
a union of embedded manifolds by Lojasiewicz's structure theorem for real-analytic sets \cite{krantzparks}, \cite{lawton2}, \cite{lojasiewicz}. Since $\theta(\mathbb R)$ is a uniformly distributed embedding, if the dimension of $\mathcal S(A)$ were
less than $d-1$ then $\underline d(\mathcal S(\widehat a)) = 0.$
\end{proof}
\begin{theorem}\label{norms}
 If $\alpha$ is a PV number of degree $d \geq 2$ then the set of norms $N(\Lambda_\alpha)$ is a set of rational numbers whose denominators
 have only a finite number of prime divisors and whose numerators are values of integral forms $($homogeneous polynomials$)$ of degree $d$
 in $d$ integer variables. The number of these integer values having modulus $\leq L$ is asymptotically bounded below by
  $O(L^{2/d})$ for $d \geq 3$ and by $O(L/(\log L)^p)$ for some $p \in (0,1)$ for $d = 2.$
\end{theorem}
\begin{proof}
Let $\alpha = \alpha_1,...\alpha_d$ be the Galois conjugates of $\alpha.$
Theorem \ref{pisot} implies that $\lambda \in \Lambda_\alpha$ iff
there exist $m \in \mathbb Z$ and $\mu_k \in \mathbb Q[\alpha_k)$ such that
$\lambda = \mu_1 \alpha^j$ where $[\mu_1,...,\mu_d]^T \in V^{-1}\mathbb Z^d.$
The elements of the $k$th row of $V^{-1},$ being the coefficients of the Lagrange interpolating polynomial
$Q_k(X) = \frac{P_{\alpha}(X)}{(X-\alpha_k)P^{\prime}(\alpha_k)},$ belong to the field $\mathbb Q[\alpha_k]$ 
and the elements in every column of $V^{-1}$ are Galois conjugates. Therefore $N(\lambda) = N(\alpha)^m\, N(\mu_1),$ 
and $N(\mu_1) = \prod_{k=1}^d \mu_k$ is a form with rational coefficients of degree $d$ in $d$ integer variables
$($the coordinates of $\mathbb Z^d).$ The denominators of the coefficients of the form $N(\mu_1)$ must divide $\det V$ so
the prime factors of the denominators of the numbers in $N(\Lambda_\alpha)$ must divide $N(\alpha)$ or $\det V.$ Therefore
a positive fraction of the numbers in $N(\Lambda_\alpha)$ have numerators that are values of an integral form of degree $d$ in $d$
integer variables. For $d \geq 3$ we obtain a binary form of degree $3$ by setting all except $2$
of these integer variables to $0,$ and obtain the lower asymptotic bound for the values of the numerators 
by a Theorem of Erd\"{o}s and Mahler \cite{erdosmahler}.
For $d = 2$ we obtain a binary quadratic form and we obtain a lower asymptotic bound 
given by a Theorem of Odoni (\cite{odoni}, Theorem S).
\end{proof}
We refer the reader to Section 4 for a discussion of Pontryagin duality. If $\chi \in \widehat G$ and $c \in \mathbb T_c$ then the zero set
$\mathcal S(\chi - c) = \{g \in G:\chi(g) - c = 0 \}$ has dimension $d-1.$
We call a real-analytic subset of $G$ ${\it simple}$ if it is contained
in a finite union of such sets. Lagarias and Yang conjectured \cite{lagariaswang} that certain real-analytic subsets of $\mathcal T^n,$ that arise in the construction of refinable functions of several variables related to tilings and that are analogous to our set $\mathcal S(A),$ are simple. We used Lojasiewicz's theorem \cite{lawton2} to prove their conjecture. Thus we find the following result interesting:
\begin{theorem}\label{simple}
 If $A$ is nonzero then $\mathcal S(A)$ is not simple.
\end{theorem}
\begin{proof}
The argument used in the proof of Theorem \ref{lowerbound} shows that
for every $\lambda \in \Lambda_\alpha$ there exists $m \in \mathbb Z$ such that $\lambda\alpha^m \in \mathcal S(\widehat a).$
Therefore Theorem \ref{norms} implies that the set of norms $N(\mathcal S(\widehat a))$ contains a set of rational numbers whose
numerators whose modulus is less than $L$ has asymptotic density $> O(L^{1/d}).$ If $\mathcal S(A)$ were a proper simple subset of $G$
then all but a finite number of points in $\mathcal S(\widehat a)$ would be contained in a finite union of sets having the form
$\beta + \delta \mathbb Z$ where $\beta$ and $\delta$ are elements in $\mathbb Q[\alpha].$ However $N(\beta + \delta k)$ is
a form of degree $d$ with rationall coefficients in single integer variable $k$ and therefore the set of numerators of the
values of this form has asymptotic density $O(L^{1/d}),$ thus giving a contradiction.
\end{proof}
\begin{remark}
Integral binary quadratic forms were studied by Gauss \cite{gauss}, who focussed on forms having negative discriminant.
Asymptotic estimates for the number of integers represented by integral binary quadratic forms with negative discriminant
were obtained in the doctoral dissertation of Bernays \cite{bernays} and by James \cite{james}. Numerical studies were compiled
by Sloan \cite{sloan}.
\end{remark}
Inspired by developments in Diophantine geometry and o-minimal theory we make the following assertion
whose validity proves Conjecture \ref{conj}.
\begin{conjecture}
Every real-analytic subset of $G$ that intersects every homoclinic orbit is simple.
\end{conjecture}
\section{Appendix: Pontryagin Duality and the Kronecker-Weyl Theorem}
\label{app}
A {\it character} of a locally compact abelian topological group $G$ is a continuous homomorphism $\chi : G \rightarrow \mathbb T_c.$ The
{\it dual group} $\widehat G$ consists of all characters under pointwise multiplication and the topology of uniform convergence on compact
subsets.
The {\it Pontryagin duality} theorem says that the homomorphism $\gamma : G \rightarrow \widehat {\widehat G}$
\begin{equation}\label{nat}
  \gamma(g)(\chi) = \chi(g), \ \ g \in G, \ \chi \in \widehat G
\end{equation}
is a bijective isomorphism. This was proved for second countable groups that are either compact or discrete in 1934 by
Lev Semyonovich Pontryagin \cite{pontrjagin}  and extended to general locally compact groups in 1934 by Egbert van Kampen \cite{vankampen}.
This theory shows that $G$ is compact; discrete; connected; dimension $d$  iff $\widehat G$ is discrete; compact; torsion free; rank $d.$ For $a \in \mathbb Z^n$ define $\chi_a \in \widehat {\mathbb T^n}$ by
$\chi_a(g) = \exp 2\pi i (a(1)g(1) + \cdots + a(n)g(n)).$
The mapping $a \rightarrow \chi_a$ is an isomorphism.
If $H$ is a closed subgroup of $G$ then the quotient $G/H$ is locally compact and we have an obvious injective homomorphism $\widehat {G/H} \rightarrow \widehat G.$ Therefore Pontrygin duality implies that $H$ is proper iff every character on $G$ that vanishes on $H$ vanishes on $G.$ This gives the following classical result \cite{weyl}.
\begin{lemma}\label{Weyl}(Kronecker-Weyl Theorem)
    If $v \in \mathbb R^n$ then $\mathbb R v + \mathbb Z^n$ is dense in $\mathbb T^n$ iff the entries of $v$ are rationally independent.
\end{lemma}
\begin{proof}
The closure $H = \overline {\mathbb Rv + \mathbb Z^n}$ is a closed subgroup of $\mathbb T^n$ and is proper iff there exists $a \in \mathbb Z^n\backslash \{0\}$ with
$
    \chi_a(tv + \mathbb Z^n) = \exp (2\pi i ta^{T}v), \ t \in \mathbb R,
$
or equivalently, $a^{T}v = 0.$ This occurs iff the entries of $v$ are rationally dependent.
\end{proof}
\section*{Acknowledgement} We thank Keith Mathews and John Robertson for helpful discussions about the representation of
integers by integral binary quadratic forms.


\begin{thebibliography}{10}
%
\bibitem{bernays} P. Bernays, {\it \"{U}ber die Darstellung von positiven, ganzen Zahlen durch die primitiven
bin\"{a}ren quadratischen Formen einer nichtquadratischen Diskriminante}, Dissertation, Universit\"{a}t G\"{o}ttingen,
1912.
%
\bibitem{bohr} H. Bohr, {\it Zur Theorie der fastperiodischen Funktionen I}. Acta Mathematica, {\bf 45} (1924) 29--127.
%
\bibitem{cassels} J. W. S. Cassels, {\it An Introduction to Diophantine Approximation}, Cambridge Tracts in Mathematics
    and Mathematical Physics No. 45, Cambridge University Press, 1957.
%
\bibitem{CDM} A. S. Cavaretta, W. Dahmen, C. A. Michelli,
\textit{Stationary Subdivision}, Memoirs of the American Mathematical Society, 1991.
%
\bibitem{daifengwang} X. R. Dai, D. J. Feng, Y.Wang, {\it Refinable functions with non-integer dilations},
Journal of Functional Analysis, {\bf 250} (2007) 1--20.
%
\bibitem{daub} I. Daubechies,
\textit{Orthonormal bases of compactly supported wavelets},
Communications on Pure and Applied Mathematics, {\bf 41} (7) (1988) 909--966.
%
\bibitem{dedekind} R. Dedekind, {\it  \"{U}ber den Zusammenhang zwischen der Theorie der Ideale und der Theorie der h\"{o}heren Congruenzen},
Abhandlungen der K\"{o}niglichen Gesellschaft der Wissenschaften zu G\"{o}ttingen {\bf 23} (1) (1878) 3-37.
%
\bibitem{erdos} P. Erd\"{o}s, {\it On a family of symmetric Bernoulli convolutions},
American Journal of Mathematics, {\bf 61} (1939) 974--976.
%
\bibitem{erdosmahler} P. Erd\"{o}s and K. Mahler, {\it On the number of integers
which can be represented by a binary form}, Journal of the London Mathematical
Society, {\bf 13} (1938) 134-139.
%
\bibitem{fell} J. M. G. Fell, {\it A Hausdorff topology for the
closed sets of a locally compact non-Hausdorff space}, Proceedings
of the American Mathematical Society, {\bf 13} $($1962$)$ 472--476.
%
\bibitem{furstenberg} H. Furstenberg, {\it Recurrence in Ergodic Theory and Combinatorial Number Theory},
Princeton University Press, 1981.
%
\bibitem{gauss} C. F. Gauss, {\it Disquisitiones Arithmeticae}, Yale University Press, 1965.
%
\bibitem{gazeauspatera} J. P. Gazeau and J. Patera, {\it Tau-wavelets of Haar}, Journal of Physics A: Math. Gen.,
{\bf 29} (1996) 4549-4559.
%
\bibitem{gazeauspiridonov} J. P. Gazeau and V. Spiridonov, {\it Toward discrete wavelets with irrational scaling factor},
Journal of Mathematical Physics, {\bf 37} (6) (1996) 3001-3013.
%
\bibitem{ghm} J. S. Geronimo, D. P. Hardin, and P. R. Massopust, {\it Fractal functions
and wavelet expansions based on several scaling functions}, Journal of Approximation
Theory, {\bf 78} (1994) 373--401.
%
\bibitem{james} R. D. James, {\it The distribution of integers represented by quadratic forms},
American Journal of Mathematics, {\bf 60} (3) (1938) 737-744.
%
\bibitem{krantzparks}  S. G. Krantz and H. R. Parks, {\it A Primer of Real
Analytic Functions}, Birkhäuser, Boston, 1992.
%
\bibitem{lagariaswang}  J. C. Lagarias and Y. Wang, {\it Integral self-affine tiles
in $R^n$, part II: lattice tilings}, Journal of Fourier Analysis and Applications {\bf 3} $($1$)$ $($1997$)$ 83-102.
%
\bibitem{lam} P. F. Lam, {\it On expansive transformation groups}, Transactions of the American Mathematical Society,
{\bf 150} $($1970$)$ 131-138.
%
\bibitem{lawton1} W. Lawton, {\it The structure of compact connected groups which admit an expansive
automorphism}, Recent advances in Topological Dynamics, Lecture Notes in Mathematics, vol. 318, Springer-Verlag,
Berlin-Heidelberg-New York, 1973, pp. 182--196.
%
\bibitem{lawton2}  W. Lawton, {\it Proof of the hyperplane zeros conjecture of Lagarias and Wang}, The Journal of Fourier
Analysis and Applications, {\bf 14} $($4$)$ $($2008$)$ 588--605.
%
\bibitem{lawton3} W. Lawton, {\it Multiresolution analysis on quasilattices}, Poincare Journal
of Analysis $\&$ Applications, {\bf 2} $($2015$)$ 37-52.
%
\bibitem{lindward} D. A. Lind and T. Ward, {\it Automorphisms of solenoids and $p$-adic entropy},
Ergodic Theory $\&$ Dynamical Systems, {\bf 8} $($1988$)$ 411-419.
%
\bibitem{lojasiewicz} S. Lojasiewicz, {\it Introduction to Complex Analytic Geometry}, Birkh\"{a}user, Boston, 1991.
%
\bibitem{narkiewicz} W. Narkiewicz, {\it Elementary and Analytic Theory of Algebraic Numbers}, Springer Monographs in
Mathematics (3 ed.), Berlin, 2004.
%
\bibitem{odoni} R. W. K. Odoni, {\it Representations of algebraic integers by
binary quadratic forms and norm forms from full modules of extension fields},
Journal of Number Theory, {\bf 10} $($1978$)$ 324--333.
%
\bibitem{pisot} C. Pisot, {\it La r\'{e}partition modulo 1 et nombres r\'{e}els alg\'{e}briques}, Ann. Sc. Norm. Super. Pisa, II, Ser. 7 $($1938$)$ 205--248.
%
\bibitem{pontrjagin} L. Pontrjagin, {\it The Theory of Topological Commutative Groups},
Annals of Mathematics, {\bf 35}$($2$)$ $($1934$)$ 361-388.
%
\bibitem{schmidt} K. Schmidt, {\it Dynamical Systems of Algebraic Origin}, Birkh\"{a}ser, Basel, 1995.
%
\bibitem{sloan} N. J. A. Sloane, {\it Binary Quadratic Forms and OEIS}, Webpage \\
https:$//$oeis.org$/$wiki$/$Binary$\_$Quadratic$\_$Forms$\_$and$\_$OEIS,
Created June 05 - 12, 2014.
%
\bibitem{tychonoff} Q. N. Tychonoff, {\it \"{U}ber die topologische Erweiterung von R\"{a}umen",
Mathematische Annalen}, {\bf 102} (1): (1930) 544–561,
%
\bibitem{vankampen} E. R. van Kampen, {\it Locally compact abelian groups}, Proceedings of the National Academy of Science, (1934) 434-436.
%
\bibitem{weyl} H. Weyl, {\it \"{U}ber die Gleichverteilung von Zahlen mod. Eins}, Math. Ann.,
{\bf 77} (1916) 313-352.
%
\end{thebibliography}
\end{document}